\documentclass[reqno]{amsart}

\usepackage{amsmath,amsthm, txfonts, graphicx}

\newtheorem{theorem}{Theorem}[section]

\newtheorem{corollary}[theorem]{Corollary}

\numberwithin{equation}{section}

\theoremstyle{definition}
\newtheorem{definition}[theorem]{Defintion}
\newtheorem{remark}[theorem]{Remark}

\DeclareMathOperator{\dom}{dom}%

\newcommand{\vstr}[1][3]{\rule{0ex}{#1ex}}
\newcommand{\negsp}[1][20]{\mspace{-#1mu}}
\newcommand{\evald}[2][]{\ensuremath{\negsp[4]\left.\vstr[2.0] \right|_{#2}^{#1}}} 

\def\Tri#1#2#3#4{{ 
\put(#1,#2){\line(3,5){#3}} \put(#1,#2){\line(3,0){#4}} \count223=#1 \advance\count223 by #4
\put(\count223,#2){\line(-3,5){#3}}}}

\def\Spic#1#2#3#4{{
\count205=#2\count206=#2\count207=#2\count208=#2\count209=#2
\count210=#1\count211=#3\count214=#3\count212=#4 \divide\count210 by 2 {\ifnum\count210>0
\Spic{\count210}{#2}{#3}{#4} \multiply\count207 by 3 \multiply\count208 by 6 \multiply\count209 by
5 \multiply\count207 by \count210 \multiply\count208 by \count210 \multiply\count209 by \count210
{\advance\count211 by \count207 \advance\count212 by \count209
\Spic{\count210}{#2}{\count211}{\count212}} {\advance\count214 by \count208
\Spic{\count210}{#2}{\count214}{#4}} \else \multiply\count205 by 3 \multiply\count206 by 6
\Tri{#3}{#4}{\count205}{\count206} \fi }}}

\begin{document}

\title{Gradients of Laplacian Eigenfunctions on the Sierpinski Gasket \footnote{MSC Primary 28A80; Secondary 33E30}}
\author{Jessica L. DeGrado}
\address{\noindent Department of Mathematics\\ Cornell University\\ Ithaca, NY 14853-4201\\USA.}
\email{jld69@cornell.edu}
\thanks{Research supported by the National Science Foundation through the Research
Experiences for Undergraduates (REU) Program at Cornell University.}
\author{Luke G. Rogers}
\address{\noindent Department of Mathematics\\University of Connecticut\\Storrs CT 06269-3009\\USA.}
\email{rogers@math.uconn.edu}
\author{Robert S. Strichartz}
\address{\noindent Department of Mathematics\\ Cornell University\\ Ithaca, NY 14853-4201\\USA.}
\email{str@math.cornell.edu}
\thanks{Research supported in part by the National Science Foundation, Grant
DMS-0652440.}

\date{\today}

\begin{abstract}
We use spectral decimation to provide formulae for computing the harmonic gradients of Laplacian
eigenfunctions on the Sierpinski Gasket.  These formulae are given in terms of special functions
that are defined as infinite products.
\end{abstract}

\maketitle


\section{Introduction}

There are few functions more ubiquitous in Euclidean analysis than the sine, cosine and
exponential, which are the eigenfunctions of the Laplacian on an interval in $\mathbb{R}$. In the
theory of analysis on fractals, the Laplacian eigenfunctions arguably have an even more prominent
role, as the Laplacian is the fundamental differential operator on which the analysis is based.
Despite this, there are a number of interesting open questions about the structure of such
eigenfunctions. In this paper we consider the local behavior of Laplacian eigenfunctions on the
Sierpinski Gasket ($SG$), in terms of the harmonic tangents and gradients introduced by Teplyaev
in~\cite{MR1761365}. Using the spectral decimation method of Fukashima and Shima~\cite{MR1245223}
(see also Chapter 3 of~\cite{Strichartzbook}) we give infinite product formulae for the tangents at
boundary points, and use them to describe the one-sided tangents at junction points. These results
may be seen as a Sierpinski Gasket version of the well known formulae for the derivatives of the
sine and cosine functions on an interval, though the precise analogue on $[0,1]$ is more
complicated (see Equations~\ref{intervalcaseone}-\ref{intervalcasethree}).

The Sierpinski Gasket is the simplest non-trivial example of a fractal to which the standard theory
of analysis on fractals applies.  We refer to the monographs~\cite{Kigamibook,Strichartzbook} for
detailed proofs of all results we use from this theory. Recall that $SG\subset\mathbb{R}^{2}$ is
the attractor of an iterated function system consisting of three maps $F_{i}(x)=(x+q_{i})/2$, where
the points $q_{0}$, $q_{1}$ and $q_{2}$ are the vertices of an equilateral triangle.  This means
that $SG=\cup_{i} F_{i}(SG)$, where the sets $F_{i}(SG)$ are usually referred to as $1$-cells.  For
a length $m$ word $w=w_{1}\dotsc w_{m}$ with letters $w_{j}\in \{0,1,2\}$ we define
$F_{w}=F_{w_{1}}\circ\dotsm\circ F_{w_{m}}$ and call $F_{w}(SG)$ an $m$-cell.   The points $q_{i}$,
$i=0,1,2$ are the boundary of $SG$; the set of boundary points is $V_{0}$, and we use $V_{m}$ to
denote points of the form $F_{w}(q_{i})$ where $w$ is a word of length $m$.  These $V_{m}$ are
vertices of the usual graph approximation of $SG$ at scale $m$, in which vertices $x$ and $y$ are
joined by an edge (written $x\sim_{m}y$) if they belong to a common $m$-cell.  Clearly
$V_{\ast}=\cup_{m}V_{m}$ is dense in $SG$.

The Laplacian $\Delta$ on $SG$ is a renormalized limit of graph Laplacians $\Delta_{m}$ on the
$m$-scale graphs:
\begin{gather}
    \Delta u(x) = \frac{3}{2}\lim_{m\rightarrow\infty} 5^{m} \Delta_{m}u (x) \label{intro_defnofDelta}\\
    \Delta_{m} u (x)  = \sum_{y\sim_{m}x} \bigl( u(y)-u(x) \bigr) \quad \text{for $x\in V_{m}\setminus
    V_{0}$}
    \end{gather}
and we say $u\in\dom(\Delta)$ if the right side of \eqref{intro_defnofDelta} converges uniformly on
$V_{\ast}\setminus V_{0}$ to a continuous function.  The function is extended to $SG$ by continuity
and density of $V_{\ast}$.  At a boundary point $q_{i}\in V_{0}$ there is an associated normal
derivative defined (with $q_{i+3}=q_{i}$) by
\begin{equation*}
    \partial_{n}u(q_{i})
    =\lim_{m\rightarrow\infty} \Bigl(\frac{5}{3}\Bigr)^{m}
    \bigl( 2u(q_{i})-u(F^{m}_{i}(q_{i+1}))-u(F^{m}_{i}(q_{i+2})) \bigr).
    \end{equation*}

A harmonic function $h$ is one for which $\Delta h=0$, and for any assignment of values on $V_{0}$
there is a unique harmonic function with these boundary values.  For this reason we identify the
harmonic functions with the space of functions on $V_{0}$. The harmonic functions are also graph
harmonic, so it is elementary to compute the values of $h$ on $V_{1}$ from those on $V_{0}$, and
recursively to obtain the values on $V_{m}$ for any $m$.  It will be useful to formalize this by
defining the harmonic extension matrices $A_{i}$, which map the values of $h$ on $V_{0}$ to those
on $F_{i}(V_{0})$, by
\begin{equation*}
    \begin{pmatrix}
        h\circ F_{i}(q_{0}) \\
        h\circ F_{i}(q_{1}) \\
        h\circ F_{i}(q_{2})
        \end{pmatrix}
    = A_{i} \begin{pmatrix}
        h(q_{0})\\
        h(q_{1})\\
        h(q_{2})
        \end{pmatrix}
    \end{equation*}
and more generally $\bigl( h\circ F_{w}(q_{0}),h\circ F_{w}(q_{1}),h\circ F_{w}(q_{2}) \bigr)^{T} =
A_{w} \bigl( h(q_{0}),h(q_{1}),h(q_{2})\bigr)^{T}$, where $A_{w}=A_{w_{m}}A_{w_{m-1}}\dotsm
A_{w_{1}}$.  We usually write this in the compact form
\begin{equation*}
    h \evald{F_{w}V_{0}} = A_{w} h\evald{V_{0}}.
    \end{equation*}
The matrices are
\begin{equation*}
    A_{0}=\frac{1}{5}
        \begin{pmatrix}
        5&0&0\\
        2&2&1\\
        2&1&2
        \end{pmatrix},
    \quad
    A_{1}=\frac{1}{5}
        \begin{pmatrix}
        2&2&1\\
        0&5&0\\
        1&2&2
        \end{pmatrix},
    \quad
    A_{2}=\frac{1}{5}
        \begin{pmatrix}
        2&1&2\\
        1&2&2\\
        0&0&5
        \end{pmatrix}.
    \end{equation*}

The structure of eigenfunctions of the Laplacian is similar to that of the harmonic functions.
Specifically, it is true on $SG$ that if $m$ is sufficiently large then the restriction of a
function $u$ satisfying $-\Delta u=\lambda u$ from $SG$ to $V_{m}$ gives an eigenfunction of the
graph Laplacian $-\Delta_{m} u = \lambda_{m}u$, with
\begin{gather}
    \lambda_{m}(5-\lambda_{m}) = \lambda_{m-1} \label{intro_recursionforlambdam}\\
    \lambda = \frac{3}{2} \lim_{m\rightarrow\infty} 5^{m} \lambda_{m}.
    \label{intro_lambdafromlimlambdam}
    \end{gather}
Note that \eqref{intro_recursionforlambdam} implies that $\lambda_{m}$ is one of
$\frac{1}{2}(5\pm\sqrt{25-4\lambda_{m-1}})$, but the positive root is only permitted to occur for
finitely many values of $m$ in order that the limit in \eqref{intro_lambdafromlimlambdam} exists.
This {\em spectral decimation} property was first recognized by Fukashima and
Shima~\cite{MR1245223}.  It is not true on all fractals, but on those where it is valid, it gives
both a method for computing the spectrum and a recursion for the eigenfunctions~\cite{MR1997913}.
Let us define
\begin{align}
    A_{0}(\lambda) &=\frac{1}{(5-\lambda)(2-\lambda)}
        \begin{pmatrix}
        (5-\lambda)(2-\lambda)&0&0\\
        (4-\lambda)&(4-\lambda) &2\\
        (4-\lambda)&2&(4-\lambda)
        \end{pmatrix} \label{intro_Azerolambda}\\
    A_{1}(\lambda)&=\frac{1}{(5-\lambda)(2-\lambda)}
        \begin{pmatrix}
        (4-\lambda)&(4-\lambda)&2\\
        0&(5-\lambda)(2-\lambda)&0\\
        2&(4-\lambda)&(4-\lambda)
        \end{pmatrix} \label{intro_Aonelambda}\\
    A_{2}(\lambda)&=\frac{1}{(5-\lambda)(2-\lambda)}
        \begin{pmatrix}
        (4-\lambda)&2&(4-\lambda)\\
        2&(4-\lambda)&(4-\lambda)\\
        0&0&(5-\lambda)(2-\lambda) \label{intro_Atwolambda}
        \end{pmatrix}
    \end{align}
provided $\lambda\neq 2,5$.  The essence of the spectral decimation method on $SG$ may then be
summarized in the following theorem, which we have taken from Sections 3.2 and 3.3
of~\cite{Strichartzbook}.

\begin{theorem}[Spectral Decimation Method]\label{intro_nonDirefns}
If  $(\Delta+\lambda)u=0$ then there is a sequence $\{\lambda_{m}\}_{m\geq m_{0}}$ satisfying
\eqref{intro_recursionforlambdam} and \eqref{intro_lambdafromlimlambdam}, and such that
$\lambda_{m}\neq2,5,6$ for $m>m_{0}$, with the property that
$(\Delta_{m}+\lambda_{m})u\evald{V_{m}}=0$ for all $m\geq m_{0}$. The values of $u$ on $V_{m}$ for
$m>m_{0}$ can be constructed recursively using the matrices from \eqref{intro_Azerolambda},
\eqref{intro_Aonelambda}, and \eqref{intro_Atwolambda} as follows.  If $w=w_{1}w_{2}\dotsc w_{m}$
then
\begin{equation}\label{intro_spectdecfmla}
    u\evald{F_{w}(V_{0})} = A_{w_{m}}(\lambda_{m})A_{w_{m-1}}(\lambda_{m-1})\dotsm A_{w_{m_{0}+1}}(\lambda_{m_{0}+1}) u\evald{F_{w_{1}\dotsc
    w_{m_{0}}}(V_{0})},
    \end{equation}
and we call this the spectral decimation relation.

If $\lambda$ is not a Dirichlet eigenvalue then we may assume $m_{0}=0$, at which point the
condition $(\Delta_{m_{0}}+\lambda_{m_{0}})u\evald{V_{m_{0}}}=0$ is taken to be vacuous.  The
corresponding eigenspace is $3$-dimensional and parametrized by the values of $u$ on $V_{0}$.

If $\lambda$ is Dirichlet several possibilities occur.  We indicate the initial configurations, all
of which may then be continued by the spectral decimation formula.  A spanning set for the
configurations when $m_0=1$ are shown in Figure~\ref{twoandfiveseries}. The one on the left has
$\lambda_{1}=2$ while those on the right have $\lambda_{1}=5$.  If $m_{0}\geq 2$ then
$\lambda_{m_{0}}=5$ or $\lambda_{m_{0}}=6$, and in the latter case $\lambda_{m_{0}+1}=3$.  Those
with $\lambda_{m_0}=5$ are formed from scaled and rotated copies of the functions on the right in
Figure~\ref{twoandfiveseries}, arranged so that their normal derivatives cancel. A basis of chains
for $m_{0}=2$ is shown in Figure~\ref{localizedfives}; those for general $m$ are naturally indexed
by the loops in $V_{m}$, plus two strands connecting points of $V_{0}$.  In the case
$\lambda_{m_0}=6$ the eigenfunctions are indexed by points in $V_{m_{0}-1}\setminus V_{0}$; a basis
is obtained by scaling and rotating two copies of the function on the left in
Figure~\ref{sixseries} and gluing them at the chosen point, as shown on the right in
Figure~\ref{sixseries} for the case $m_{0}=2$ and a point in $V_{1}$.
\end{theorem}
\begin{figure}[htb]
\begin{picture}(105.6,90)(0,0)
\setlength{\unitlength}{.23pt} \Spic{2}{32}{0}{0} \put(183,-40){$1$} \put(382,-40){$0$}
\put(-16,-40){$0$} \put(58,152){$1$} \put(186,330){$0$} \put(300,152){$1$}
\end{picture}
\begin{picture}(105.6,90)(0,0)
\setlength{\unitlength}{.23pt} \Spic{2}{32}{0}{0} \put(183,-40){$1$} \put(382,-40){$0$}
\put(-16,-40){$0$} \put(40,152){$-1$} \put(186,330){$0$} \put(300,152){$0$}
\end{picture}
\begin{picture}(105.6,90)(0,0)
\setlength{\unitlength}{.23pt} \Spic{2}{32}{0}{0} \put(183,-40){$1$} \put(382,-40){$0$}
\put(-16,-40){$0$} \put(50,152){$0$} \put(186,330){$0$} \put(308,152){$-1$}
\end{picture}
\caption{Dirichlet eigenfunctions with $m_0 = 1$.}\label{twoandfiveseries}
\end{figure}
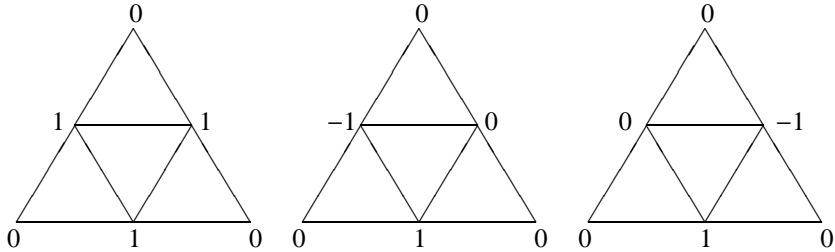
\begin{figure}[htb]
\begin{picture}(105.6,90)(0,0)
\setlength{\unitlength}{.23pt} \Spic{4}{16}{0}{0} \put(182,-40){$0$} \put(382,-40){$0$}
\put(-18,-40){$0$} \put(64,150){$0$} \put(184,330){$0$} \put(300,150){$0$} \put(210,70){$0$}
\put(152,70){$1$} \put(170,126){$-1$} \put(112,228){$0$} \put(248,228){$1$} \put(15,70){$0$}
\put(345,70){$0$} \put(72,-40){$-1$} \put(278,-40){$0$}
\end{picture}
\begin{picture}(105.6,90)(0,0)
\setlength{\unitlength}{.23pt} \Spic{4}{16}{0}{0} \put(182,-40){$0$} \put(382,-40){$0$}
\put(-18,-40){$0$} \put(64,150){$0$} \put(184,330){$0$} \put(300,150){$0$} \put(210,70){$0$}
\put(152,70){$0$} \put(184,126){$0$} \put(112,228){$1$} \put(248,228){$-1$} \put(-6,70){$-1$}
\put(345,70){$1$} \put(85,-40){$1$} \put(265,-40){$-1$}
\end{picture}
\begin{picture}(105.6,90)(0,0)
\setlength{\unitlength}{.23pt} \Spic{4}{16}{0}{0} \put(182,-40){$0$} \put(382,-40){$0$}
\put(-18,-40){$0$} \put(64,150){$0$} \put(184,330){$0$} \put(300,150){$0$} \put(210,70){$1$}
\put(152,70){$0$} \put(170,126){$-1$} \put(112,228){$1$} \put(248,228){$0$} \put(15,70){$0$}
\put(345,70){$0$} \put(85,-40){$0$} \put(265,-40){$-1$}
\end{picture}
\caption{A basis of chains for $m_0 = 2$ and $\lambda_{m_0} = 5$} \label{localizedfives}
\end{figure}
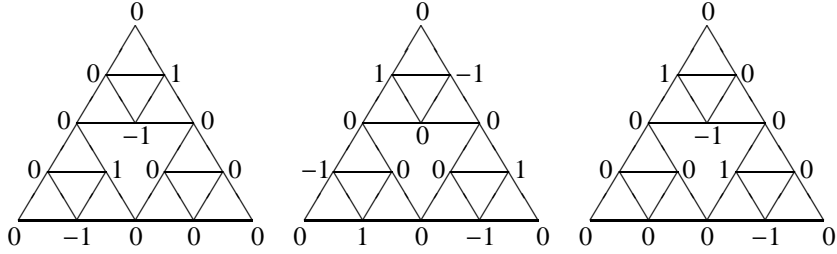
\begin{figure}[htb]
\centerline{{ \small
\begin{picture}(105.6,90)(0,0)
\setlength{\unitlength}{.23pt} \Spic{2}{32}{0}{0} \put(170,-40){$-1$} \put(382,-40){$0$}
\put(-18,-40){$2$} \put(46,152){$-1$} \put(186,330){$0$} \put(297,152){$1$}
\end{picture}
\hspace{5em}
\begin{picture}(105.6,90)(0,0)
\setlength{\unitlength}{.23pt} \Spic{4}{16}{0}{0} \put(182,-40){$2$} \put(382,-40){$0$}
\put(-18,-40){$0$} \put(68,150){$0$} \put(184,330){$0$} \put(300,150){$0$} \put(194,72){$-1$}
\put(146,72){$-1$} \put(184,126){$0$} \put(112,228){$0$} \put(248,228){$0$} \put(18,72){$1$}
\put(345,72){$1$} \put(72,-40){$-1$} \put(265,-40){$-1$}
\end{picture}
}}
\caption{Eigenfunction construction in the case $\lambda_{m_0} = 6$.}\label{sixseries}
\end{figure}
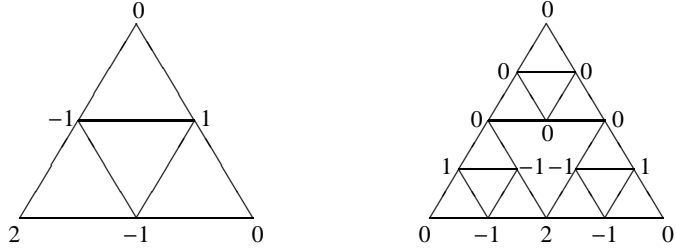

\section{Tangents to Eigenfunctions}

In~\cite{MR1761365}, Teplyaev introduced the notion of a harmonic gradient and harmonic tangent for
functions on the Sierpinski Gasket.  He also proved that functions in the domain of the Laplacian
have harmonic gradients at all points in a set of full measure, and that functions with H\"{o}lder
continuous Laplacian have harmonic gradients at all points.  Eigenfunctions of the Laplacian fall
into the latter category because of the well-known fact that continuity of $\Delta^{2}u$ implies
H\"{o}lder continuity of $\Delta u$ (\cite{Kigamibook}, Lemma 2.2.5).

\begin{definition}[\protect{\cite{MR1761365}, Section 3}]
Let $w=w_{1}w_{2}\dotsc$ be an infinite word, and $[w]_{m}=w_{1}w_{2}\dotsc w_{m}$ be the length
$m$ truncation of $w$.  If $u$ is a function on $SG$ we let $H_{[w]_{m}}u$ be the harmonic function
on $SG$ which coincides with $u$ on $F_{[w]_{m}}(V_{0})$, so
\begin{equation}\label{tangents_fmlaforharmonicapprox}
    H_{[w]_{m}}u= A_{[w]_{m}}^{-1} u\evald{F_{[w]_{m}}(V_{0})}
    = A_{w_{1}}^{-1}\dotsm A_{w_{m}}^{-1} u\evald{F_{[w]_{m}}(V_{0})}.
    \end{equation}
Define the harmonic tangent $T_{w}u$ of $u$ at $w$ to be $\lim_{m\rightarrow\infty} H_{[w]_{m}}u$,
if the limit exists.  It should be remarked that there can be two words $w$ and $w'$ such that
$F_{w}(SG)=F_{w'}(SG)$.  We will nonetheless treat the tangents $T_{w}$ and $T_{w'}$ separately, as
they are rarely equal.

The harmonic gradient is defined in the same way, but using the space of harmonic functions with
average zero and the projection of the action of the matrices $A_{i}$ to this subspace.  It is
evident that if $u$ is continuous then the gradient exists whenever the tangent exists, and
conversely.  See~\cite{MR1761364} for details.
\end{definition}
If we consider the interval $[0,1]$ rather than the Sierpinski Gasket then it is clear that the
harmonic tangent at $x_0$ is the vector $\bigl(L(0), L(1)\bigr)^T$ characterizing the unique linear
function $L(x)$ having the properties $L(x_0)= f(x_0)$ and $L'(x_0) = f'(x_0)$. For an eigenvalue
$\lambda \in \mathbb{R}$ which is not equal to $\pi^2k^2$ for any $k \in \mathbb{N}$, it may
readily  be verified that the harmonic tangent is given by
\begin{equation}\label{intervalcaseone}
    M(x_0)
    \begin{pmatrix}
        f(0)\\
        f(1)
    \end{pmatrix}
\end{equation}
where
\begin{equation}\label{intervalcasetwo}
    M(x_0) = \frac{1}{\sin\sqrt{\lambda}}
    \begin{pmatrix}
        a + b && a- b\\
        a-b && a+b
    \end{pmatrix}
\end{equation}
and
\begin{equation}\label{intervalcasethree}
a =\sin\left((1-x_0)\sqrt{\lambda}\right),\quad
b=\sqrt{\lambda}x_0\cos\left((1-x_0)\sqrt{\lambda}\right).
\end{equation}
We cannot obtain as explicit a description of the harmonic tangent on SG; however, we produce
formulae that permit its computation at any point of $V_*$. The key observation is that when $u$ is
an eigenfunction, the computation of the gradient has a particularly elegant structure.  Recall
from Theorem~\ref{intro_nonDirefns} that the values of $u$ on $F_{[w]_{m}}(V_{0})$ may be computed
using the spectral decimation method, meaning that starting from a scale $m_{0}$ they can be
obtained as in~\eqref{intro_spectdecfmla}. Combining this
with~\eqref{tangents_fmlaforharmonicapprox} we see that
\begin{align}
    T_{w}u
    &= \lim_{m\rightarrow\infty} A_{w_{1}}^{-1}\dotsm A_{w_{m}}^{-1}
        \cdot A_{w_{m}}(\lambda_{m})\dotsm A_{w_{m_{0}+1}}(\lambda_{m_{0}+1})
        \, u\evald{F_{w_{1}\dotsc w_{m_{0}}}(V_{0})} \notag\\
    &= \Bigl( A_{w_{1}}^{-1}\dotsm A_{w_{m_{0}}}^{-1} \Bigr)
        \Bigl(  \lim_{m\rightarrow\infty} A_{w_{m_{0}+1}}^{-1} \dotsm A_{w_{m}}^{-1} \cdot A_{w_{m}}(\lambda_{m})\dotsm
        A_{w_{m_{0}+1}}(\lambda_{m_{0}+1}) \Bigr)
        \,  u\evald{F_{w_{1}\dotsc w_{m_{0}}}(V_{0})} \label{tangent_fmlafortangent}
    \end{align}
in which we know the limit exists by Theorem 3 of~\cite{MR1761365}. A special case occurs when
$F_{w}(SG)$ is a point in $V_{\ast}$, because in this case, all but finitely many letters in the
word $w$ are equal to a single letter $i$.  By taking $m_{0}$ to be sufficiently large we see that
it is useful to understand the limit
\begin{equation*}
    \lim_{k\rightarrow\infty}
        A_{i}^{-k} \cdot A_{i}(\lambda_{m_{0}+k})\dotsm A_{i}(\lambda_{m_{0}+1})
    \end{equation*}
and it is evident from the symmetry of the matrices $A_{i}$ and $A_{i}(\lambda)$ that it suffices
to deal with the case $i=0$.

\begin{theorem}\label{tangents_explicitformula}
Let $\alpha=(0,1,1)^{T}$, $\beta=(0,1,-1)^{T}$, $\gamma_{m}=(4,4-\lambda_{m},4-\lambda_{m})^{T}$.
If neither of the values $2$ or $5$ occur in the sequence $\{\lambda_{m}\}_{m>m_0}$, then
\begin{gather*}
    \lim_{k\rightarrow\infty}
        A_{0}^{-k} \cdot A_{0}(\lambda_{m_{0}+k})\dotsm A_{0}(\lambda_{m_{0}+1})\, \alpha
    = \frac{ 4\lambda } {3\cdot5^{m_{0}}\lambda_{m_{0}}(2-\lambda_{m_{0}+1})}
        \prod_{j=2}^{\infty} \Bigl( 1 -\frac{\lambda_{m_{0}+j}}{3} \Big)\, \alpha \\
    \lim_{k\rightarrow\infty}
        A_{0}^{-k} \cdot A_{0}(\lambda_{m_{0}+k})\dotsm A_{0}(\lambda_{m_{0}+1})\, \beta
    = \frac{2\lambda}{3\cdot5^{m_{0}}\lambda_{m_{0}}}\, \beta\\
    \lim_{k\rightarrow\infty}
        A_{0}^{-k} \cdot A_{0}(\lambda_{m_{0}+k})\dotsm A_{0}(\lambda_{m_{0}+1})\, \gamma_{m_{0}}
    = (4,4,4)^{T}
    \end{gather*}
\end{theorem}

\begin{proof}
Observe that $\alpha$ and $\beta$ are eigenvectors of $A_{0}(\lambda)$, with eigenvalues
$(6-\lambda)(2-\lambda)^{-1}(5-\lambda)^{-1}$ and $(5-\lambda)^{-1}$ respectively, provided
$\lambda\neq 2,5$.  As $A_{0}=A_{0}(0)$ is a special case, we compute immediately that
\begin{equation*}
    \lim_{k\rightarrow\infty}
        A_{0}^{-k} \cdot A_{0}(\lambda_{m_{0}+k})\dotsm A_{0}(\lambda_{m_{0}+1}) \beta
    = \lim_{k\rightarrow\infty} 5^{k} \prod_{j=1}^{k} (5-\lambda_{m_{0}+j})^{-1} \beta.
    \end{equation*}
However induction on \eqref{intro_recursionforlambdam} implies that $\lambda_{m_{0}}=
\lambda_{m_{0}+k} \prod_{j=1}^{k} (5-\lambda_{m})$ and therefore
\begin{equation*}
    \lim_{k\rightarrow\infty}
        A_{0}^{-k} \cdot A_{0}(\lambda_{m_{0}+k})\dotsm A_{0}(\lambda_{m_{0}+1}) \beta
    = \lim_{k\rightarrow\infty} 5^{k} \lambda_{m_{0}+k} \lambda_{m_{0}}^{-1} \,\beta
    = \frac{2\lambda}{3\cdot5^{m_{0}}\lambda_{m_{0}}} \beta
    \end{equation*}
The corresponding computation for $\alpha$ can be simplified by observing from
\eqref{intro_recursionforlambdam} that $(6-\lambda_{m})=(3-\lambda_{m+1})(2-\lambda_{m+1})$,
whereupon
\begin{align*}
    \lefteqn{\lim_{k\rightarrow\infty}
        A_{0}^{-k} \cdot A_{0}(\lambda_{m_{0}+k})\dotsm A_{0}(\lambda_{m_{0}+1}) \alpha }\quad & \\
    &= \lim_{k\rightarrow\infty} \Bigl(\frac{5}{3} \Bigr)^{k} \prod_{j=1}^{k}
        \frac{(6-\lambda_{m_{0}+j})}{(5-\lambda_{m_{0}+j}) (2-\lambda_{m_{0}+j})} \alpha \\
    &= \lim_{k\rightarrow\infty} \biggl( \frac{ 2-\lambda_{m_{0}+k}} {2-\lambda_{m_{0}+1}} \biggr)
        \biggl( 5^{k} \prod_{j=1}^{k} \frac{1}{5-\lambda_{m_{0}+j}} \biggr)
        \biggl( 3^{-k} \prod_{j=1}^{k} (3-\lambda_{m_{0}+j+1} ) \bigg) \alpha\\
    &= \biggl( \frac{ 2 } {2-\lambda_{m_{0}+1}} \biggr)
        \biggl( \frac{2\lambda}{3\cdot5^{m_{0}}\lambda_{m_{0}}} \biggr)
        \prod_{j=2}^{k+1} \Bigl( 1 -\frac{\lambda_{m_{0}+j}}{3} \Big)\, \alpha
    \end{align*}
where we used the previously computed limit for the middle factor, and the fact that
$\lambda_{m}=O(5^{-m})$ as $m\rightarrow\infty$ from \eqref{intro_lambdafromlimlambdam}.

For $\gamma_{m}$ the situation is a little different.  Observe that
\begin{equation*}
    A_{0}(\lambda_{m+1}) \gamma_{m}
    = (2-\lambda_{m+1})^{-1} (5-\lambda_{m+1})^{-1}
        \begin{pmatrix}
        4 (2-\lambda_{m+1}) (5-\lambda_{m+1})\\
        4(4-\lambda_{m+1}) + (4-\lambda_{m})(6-\lambda_{m+1}) \\
        4(4-\lambda_{m+1}) + (4-\lambda_{m})(6-\lambda_{m+1})
        \end{pmatrix}
    \end{equation*}
however we can perform the following simplification from \eqref{intro_recursionforlambdam}:
\begin{align*}
    4(4-\lambda_{m+1}) + (4-\lambda_{m})(6-\lambda_{m+1})
    &= 4(4-\lambda_{m+1}) + (4- 5\lambda_{m+1}+\lambda_{m+1}^{2} )(6-\lambda_{m+1})\\
    &= 4(4-\lambda_{m+1}) + (4- \lambda_{m+1})(1- \lambda_{m+1})(6-\lambda_{m+1})\\
    &= (4-\lambda_{m+1})( 10 -7\lambda_{m+1} + \lambda_{m+1}^{2})\\
    &= (4-\lambda_{m+1}) ( 5-\lambda_{m+1}) (2-\lambda_{m+1}).
    \end{align*}
Inserting this into the previous computation shows that $A_{0}(\lambda_{m+1}) \gamma_{m} =
\gamma_{m+1}$ provided $\lambda_{m}\neq2,5$, and therefore $A_{0}(\lambda_{m_{0}+k})\dotsm
A_{0}(\lambda_{m_{0}+1})\gamma_{m_{0}}=\gamma_{m_{0}+k}$.  To proceed we must apply $A_{0}^{-k}$ to
$\gamma_{m_{0}+k}$, which is most easily done by writing it in terms of eigenvectors as
$\gamma_{m_{0}+k}=(4,4,4)^{T}-\lambda_{m_{0}+k}\alpha$.  The result is
\begin{align*}
    \lim_{k\rightarrow\infty}
        A_{0}^{-k} \cdot A_{0}(\lambda_{m_{0}+k})\dotsm A_{0}(\lambda_{m_{0}+1})
        \gamma_{m_{0}}
    &= \lim_{k\rightarrow\infty} A_{0}^{-k} \Bigl( (4,4,4)^{T}-\lambda_{m_{0}+k}\alpha \Bigr)\\
    &= \lim_{k\rightarrow\infty} \biggl( (4,4,4)^{T} - \bigl(\frac{5}{3}\Bigr)^{k}
    \lambda_{m_{0}+k} \biggr)\\
    &= (4,4,4)^{T}
    \end{align*}
because $\lambda_{m}=O(5^{-m})$ as $m\rightarrow\infty$.
\end{proof}

\begin{theorem}\label{hartan0}
Suppose $w$ is a word of the form $w=[w]_{k_{0}}000\dotsm$, that $(\Delta+\lambda)u=0$, and that
$m_{0}$ is chosen large enough that the spectral decimation formula holds with
$\lambda_{m}\neq2,5,6$ when $m>m_{0}$.  If $k=\max\{k_{0},m_{0}\}$ then
\begin{equation}\label{tangentformulawithM0}
    T_{w}u
    = \Bigl( A_{w_{1}}^{-1}\dotsm A_{w_{k}}^{-1} \Bigr)
        M_{0}(\lambda,k)\, A_{w_{k}}(\lambda_{k})\dotsm A_{w_{m_{0}+1}}(\lambda_{m_{0}+1})  \, u\evald{F_{w_{1}\dotsc w_{m_{0}}}(V_{0})}
    \end{equation}
where
\begin{equation}\label{matrixformoftangentoperation}
    M_{0}(\lambda,k)
    =\begin{pmatrix}
        1 & 0& 0\\
        1- \frac{(4-\lambda_{k})\lambda\tau_{k}(\lambda)}{3\cdot5^{k}\lambda_{k}} &
        \frac{\lambda(2\tau_{k}(\lambda)+1)}{3\cdot5^{k}\lambda_{k}} & \frac{\lambda(2\tau_{k}(\lambda)-
        1)}{3\cdot5^{k}\lambda_{k}}\\
        1- \frac{(4-\lambda_{k})\lambda\tau_{k}(\lambda)}{3\cdot5^{k}\lambda_{k}} &
        \frac{\lambda(2\tau_{k}(\lambda)-1)}{3\cdot5^{k}\lambda_{k}} & \frac{\lambda(2\tau_{k}(\lambda)+
        1)}{3\cdot5^{k}\lambda_{k}}
        \end{pmatrix}
    \end{equation}
and
\begin{equation*}
    \tau_{k}(\lambda) = \frac{1}{(2-\lambda_{k+1})} \prod_{j=2}^{\infty} \Bigl( 1 -\frac{\lambda_{k+j}}{3}
    \Big).
    \end{equation*}
\end{theorem}
\begin{proof}
From \eqref{tangent_fmlafortangent} we have
\begin{align*}
    T_{w}u
    &= \Bigl( A_{w_{1}}^{-1}\dotsm A_{w_{k}}^{-1} \Bigr)
        \Bigl( \lim_{m\rightarrow\infty} A_{0}^{-(m-k)} \cdot A_{0}(\lambda_{m})
        \dotsm A_{0}(\lambda_{k+1}) \Bigr) \,  u\evald{F_{w_{1}\dotsc w_{k}}(V_{0})}  \\
    &= \Bigl( A_{w_{1}}^{-1}\dotsm A_{w_{k}}^{-1} \Bigr)
        \Bigl( \lim_{m\rightarrow\infty} A_{0}^{-(m-k)} \cdot A_{0}(\lambda_{m})
        \dotsm A_{0}(\lambda_{k+1}) \Bigr) \, A_{w_{k}}(\lambda_{k})\dotsm A_{w_{m_{0}+1}}(\lambda_{m_{0}+1})
        \, u\evald{F_{w_{1}\dotsc w_{m_{0}}}(V_{0})}
    \end{align*}
because $w_{j}=0$ for $j\geq k$ and spectral decimation applies for $j\geq m_{0}$.  The result is
therefore equivalent to
\begin{equation*}
    M_{0} = \lim_{m\rightarrow\infty} A_{0}^{-(m-k)} \cdot A_{0}(\lambda_{m}) \dotsm A_{0}(\lambda_{k+1})
\end{equation*}
which follows from Theorem \ref{tangents_explicitformula}.
\end{proof}

\begin{remark}
The function $\tau_{k}(\lambda)$ may appear to depend on the sequence $\{\lambda_{j}\}$, but in
fact this sequence is uniquely determined by $\lambda$. Indeed, there is an entire analytic
function $\Psi(z)$ with the property that $\lambda_{j}=\Psi(5^{-j}\lambda)$.  To see this, let
$\psi(z)=z(5-z)$ and $\psi_{m}(z)=\psi^{\circ m}(\frac{2}{3}5^{-m}z)$.  The sequence $\psi_{m}(z)$
consists of entire functions with $\psi_{m}(0)=0$ and $\psi_{m}'(0)=\frac{2}{3}$, so is normal with
limit $\Psi(z)$, a power series for which may be computed recursively.  It is then clear that
$\Psi(5^{-j}\lambda)=\lim_{m}\Psi(\frac{3}{2}5^{m-j}\lambda_{m})=\lim_{m}\psi_{m-j}(\lambda_{m})=\lambda_{j}$.
Moreover, we may define
\begin{equation}\label{specialfunctiondefn}
    \Upsilon(\lambda) = \frac{1}{(2-\Psi(5^{-1}\lambda))} \prod_{j=2}^{\infty} \Bigl( 1
    -\frac{\Psi(5^{-j}\lambda)}{3} \Bigr)
\end{equation}
at which point $\tau_{k}(\lambda)= \Upsilon(5^{-k}\lambda)$.  In the same way that there are
special functions associated to differential equations in Euclidean analysis, we suggest that the
functions $\Psi(\lambda)$ and $\Upsilon(\lambda)$ should be considered to be special functions in
analysis on the Sierpinski Gasket.  In terms of these functions,
\eqref{matrixformoftangentoperation} has the form
\begin{equation*}
    M_{0}(\lambda,k)
    =\begin{pmatrix}
        1 & 0& 0\\
        1-\frac{\lambda\bigl(4-\Psi(5^{-k}\lambda)\bigr)\Upsilon(5^{-k}\lambda)}{3\cdot5^{k}\Psi(5^{-k}\lambda)}
        & \frac{\lambda\bigl( 2\Upsilon(5^{-k}\lambda)+1\bigr)}{3\cdot5^{k}\Psi(5^{-k}\lambda)}
        & \frac{\lambda\bigl( 2\Upsilon(5^{-k}\lambda)- 1\bigr)}{3\cdot5^{k}\Psi(5^{-k}\lambda)}\\
        1-\frac{\lambda\bigl( 4-\Psi(5^{-k}\lambda)\bigr)\Upsilon(5^{-k}\lambda)}{3\cdot5^{k}\Psi(5^{-k}\lambda)}
        & \frac{\lambda\bigl(2\Upsilon(5^{-k}\lambda)-1 \bigr)}{3\cdot5^{k}\Psi(5^{-k}\lambda)}
        & \frac{\lambda\bigl(2\Upsilon(5^{-k}\lambda)+1 \bigr)}{3\cdot5^{k}\Psi(5^{-k}\lambda)}
        \end{pmatrix}
    \end{equation*}
\end{remark}

As a particular consequence we may compute the normal derivatives of the eigenfunctions at points
of $V_{0}$, because they are the same as the normal derivatives of the tangent functions. We expect
this observation to have applications in the construction of a resolvent for the Laplacian.

\begin{corollary}
If $(\Delta+\lambda)u=0$ and the spectral decimation formula holds with $\lambda_{m}\neq2,5,6$ for
$m>0$, then the normal derivative of $u$ at $q_{0}$ is
\begin{align*}
    \partial_{n}u(q_{0})= \Bigl( (4-\lambda_{0})u(q_{0}) -2u(q_{1})-2u(q_{2}) \Bigr)
    \frac{2\lambda\Upsilon(\lambda)}{3\lambda_{0}}.
    \end{align*}
\end{corollary}

Excluding the values $2$,$5$ and $6$ from the sequence $\{\lambda_{m}\}$ in
Theorems~\ref{tangents_explicitformula}~and~\ref{hartan0} is necessary because they occur precisely
in the Dirichlet case, where the boundary data vanishes and cannot be used to determine the
tangent. Nonetheless, Theorem~\ref{hartan0} may be applied to find tangents to Dirichlet
eigenfunctions in a simple fashion. The reason is that the description of the Dirichlet
eigenfunctions given in Theorem~\ref{intro_nonDirefns} ensures that we need only compute the
harmonic tangents of the functions in Figure~\ref{twoandfiveseries} and the left of
Figure~\ref{sixseries}.  All harmonic tangents to Dirichlet eigenfunctions are then obtained from
these by scaling and taking suitable linear combinations.

For the basic element used to construct the 6-series (Figure~\ref{sixseries}) we can directly apply
Theorem~\ref{hartan0} with $\lambda_1=6$ and $\lambda_{2}=3$. For example, if the top vertex in
Figure~\ref{sixseries} is $q_{0}$ and $w = 0\dotsm$ we have
\begin{align*}
    T_{w} u = \frac{\lambda}{9}
    \begin{pmatrix}0\\1\\-1\end{pmatrix}.
\end{align*}

To calculate the harmonic tangents of the basic element of the $2$-series (on the left in
Figure~\ref{twoandfiveseries}) we apply Theorem~\ref{hartan0} to the function shown at left in
Figure~\ref{piecesforharmonictangents}, starting the spectral decimation at each of the values
$\lambda_1 = \frac{5\pm\sqrt{17}}{2}$.  For the $5$-series there are two basic elements  (shown at
right in Figure~\ref{twoandfiveseries}) and we proceed by calculating harmonic tangents for the
initial configurations shown in the center and on the right of
Figure~\ref{piecesforharmonictangents}, starting with $\lambda_1 = \frac{5 \pm \sqrt{5}}{2}$. The
harmonic tangents of all $2$ and $5$-series eigenfunctions then coincide with scaled and rotated
copies of these pieces and their negatives, assembled in the obvious manner.

\begin{figure}[h!]
\centerline{{ \small
\begin{picture}(105.6,90)(0,0)
\setlength{\unitlength}{.23pt} \Spic{2}{32}{0}{0} \put(382,-40){$1$} \put(-18,-40){$0$}
\put(186,330){$1$}
\end{picture}
\begin{picture}(105.6,90)(0,0)
\setlength{\unitlength}{.23pt} \Spic{2}{32}{0}{0} \put(382,-40){$1$} \put(-18,-40){$0$}
\put(175,330){$-1$}
\end{picture}
\begin{picture}(105.6,90)(0,0)
\setlength{\unitlength}{.23pt} \Spic{2}{32}{0}{0} \put(382,-40){$0$} \put(-18,-40){$0$}
\put(186,330){$1$}
\end{picture}
}}
\caption{Computing harmonic tangents of $2$ and $5$-series
eigenfunctions.}\label{piecesforharmonictangents}
\end{figure}
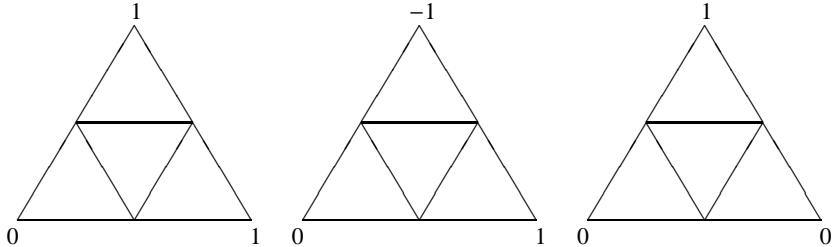

\providecommand{\bysame}{\leavevmode\hbox to3em{\hrulefill}\thinspace}
\providecommand{\MR}{\relax\ifhmode\unskip\space\fi MR }
\providecommand{\MRhref}[2]{%
  \href{http://www.ams.org/mathscinet-getitem?mr=#1}{#2}
} \providecommand{\href}[2]{#2}


\begin{thebibliography}{1}

\bibitem{MR1245223}
M.~Fukushima and T.~Shima, \emph{On a spectral analysis for the {S}ierpi\'nski
  gasket}, Potential Anal. \textbf{1} (1992), no.~1, 1--35. \MR{1245223
  (95b:31009)}

\bibitem{Kigamibook}
Jun Kigami, \emph{Analysis on fractals}, Cambridge Tracts in Mathematics, vol.
  143, Cambridge University Press, Cambridge, 2001. \MR{1840042
  (2002c:28015)}

\bibitem{MR1997913}
Leonid Malozemov and Alexander Teplyaev, \emph{Self-similarity, operators and
  dynamics}, Math. Phys. Anal. Geom. \textbf{6} (2003), no.~3, 201--218.
  \MR{1997913 (2004d:47012)}

\bibitem{MR1761364}
Robert~S. Strichartz, \emph{Taylor approximations on {S}ierpinski gasket type
  fractals}, J. Funct. Anal. \textbf{174} (2000), no.~1, 76--127. \MR{1761364
  (2001i:31018)}

\bibitem{Strichartzbook}
\bysame, \emph{Differential equations on fractals}, Princeton University Press,
  Princeton, NJ, 2006, A tutorial. \MR{2246975 (2007f:35003)}

\bibitem{MR1761365}
Alexander Teplyaev, \emph{Gradients on fractals}, J. Funct. Anal. \textbf{174}
  (2000), no.~1, 128--154. \MR{1761365 (2001h:31012)}

\end{thebibliography}
\end{document}